\documentclass[10pt]{amsart}

\usepackage{enumerate}
\usepackage{graphicx}
\usepackage{float}
\usepackage{placeins}
\usepackage{mdframed}
\usepackage{amssymb}
\usepackage{esint}
\usepackage{cool}
\usepackage[all,cmtip]{xy}
\usepackage{mathtools}
\usepackage{amstext} 
\usepackage{array}   
\usepackage[shortlabels]{enumitem}

\newcolumntype{L}{>{$}l<{$}} 
\newtheorem{theorem}{Theorem}[section]
\newtheorem{lemma}[theorem]{Lemma}
\newtheorem{cor}[theorem]{Corollary}
\newtheorem{prop}[theorem]{Proposition}
\newtheorem{setup}[theorem]{Setup}
\theoremstyle{definition}
\newtheorem{definition}[theorem]{Definition}
\newtheorem{example}[theorem]{Example}

\theoremstyle{remark}
\newtheorem{remark}[theorem]{Remark}

\newtheorem{the context}[theorem]{The Context}

\numberwithin{equation}{theorem}
\numberwithin{equation}{section}

\newcommand{\pd}{\operatorname{pd}}

\newcommand{\rank}{\operatorname{rank}}

\newcommand{\grade}{\operatorname{grade}}

\newcommand{\ext}{\operatorname{Ext}}

\newcommand{\ideal}[1]{\mathfrak{#1}}
\newcommand{\m}{\ideal{m}}

\renewcommand{\geq}{\geqslant}
\renewcommand{\leq}{\leqslant}

\renewcommand{\hom}{\Hom}

\newcommand{\Hom}{\operatorname{Hom}}

\newcommand{\maps}[5]{\xymatrix{#1 \ar[r]^-{#3} & #2 \\
#4 \ar@{|->}[r] & #5 \\}}

\newcommand{\mfa}{\mathfrak{a}}

\def\w{\wedge}

\begin{document}
\title{DG Structure on the Length 4 Big From Small Construction}

\author{Keller VandeBogert }
\date{\today}

\maketitle

\begin{abstract}
    The big from small construction was introduced by Kustin and Miller in \cite{bfs} and can be used to construct resolutions of tightly double linked Gorenstein ideals. In this paper, we expand on the DG-algebra techniques introduced in \cite{mf} and construct a DG $R$-algebra structure on the length $4$ big from small construction. The techniques employed involve the construction of a morphism from a Tate-like complex to an acyclic DG $R$-algebra exhibiting Poincar\'e duality. This induces homomorphisms which, after suitable modifications, satisfy a list of identities that end up perfectly encapsulating the required associativity and DG axioms of the desired product structure for the big from small construction.  
\end{abstract}

\section{Introduction}

Let $(R,\m)$ denote a local commutative Noetherian ring. Any quotient $R/I$ with projective dimension $2$ has minimal free resolution that admits the structure of a unique commutative associative differential graded (DG) $R$-algebra. Buchsbaum and Eisenbud (see \cite{BE}) establish a similar result for quotients $R/I$ of $R$ with projective dimension $3$, where the DG structure is no longer unique. 

For quotients $R/I$ with projective dimension at least $4$, it is no longer guaranteed that there exists an \emph{associative} DG structure on the minimal free resolution of $R/I$. A standard counterexample is given by the homogeneous minimal free resolution of the ideal $(x_1^2,x_1x_2,x_2x_3,x_3x_4,x_4^2) \subset k[x_1,x_2,x_3,x_4]$ (see \cite[Theorem 2.3.1]{inffreeres}). In the case that $R/I$ is a Gorenstein ring with projective dimension $4$, it is proved by Kustin and Miller (see \cite{char2} when characteristic $\neq 2$) and Kustin (see \cite{char3} when characteristic $\neq 3$, \cite{l4res} for a more general version) that the minimal free resolution \emph{does} admit the structure of a commutative associative DG $R$-algebra exhibiting Poincar\'e duality (see Definition \ref{dg}). For Gorenstein rings $R/I$ with projective dimension at least $5$, it is no longer guaranteed that an associative DG structure exists on the minimal free resolution (see \cite{gr5gor}). 

The existence of DG $R$-algebra structures on resolutions implies many desirable properties of the module being resolved. Indeed, the aforementioned fact that every length $3$ resolution of a cyclic module is a DG $R$-algebra is then used to deduce the well-known structure theorem for the resolution of a grade $3$ Gorenstein ideal (see \cite[Theorem 2.1]{BE}). Work of Avramov (see \cite[Section 3]{av81}) shows that if $R$ is a regular local ring and the minimal free resolution of a quotient $R/I$ is a DG $R$-algebra, then the Poincar\'e series of $R/I$ may be written in terms of the Poincar\'e series of the Koszul homology algebra $H(K^{R/I})$, which is a finite-dimensional vector space. 

There are many examples of ``canonical" resolutions admitting DG $R$-algebra structures. The Koszul complex is a DG $R$-algebra exhibiting Poincar\'e duality with product given by exterior multiplication. The standard Taylor resolution for monomial ideals, though not always minimal, is also a DG $R$-algebra (see \cite[2.1]{av81}). Similarly, it is proved in \cite{hem89} that the well-known $L$/$K$ complexes introduced by Buchsbaum and Eisenbud (see \cite{be75}) and Buchsbaum (see \cite{b78}) are DG $R$-algebras. A collection of results of a more combinatorial flavor related to the existence of DG structures on hull/Lyubeznik resolutions and resolutions of monomial ideals in general may be found in \cite{ka2019}.

In practice, there are $3$ common techniques of constructing DG structures on resolutions. The first two techniques use the fact that every resolution admits the structure of a DG $R$-algebra that is associative up to homotopy by using induced comparison maps from the so-called \emph{symmetric square} complex. For sufficiently short resolutions, these homotopies may be $0$ maps (see \cite{BE}). If not, then a modification procedure may be used to construct products that are associative (see \cite{l4res}). 

The third approach is to record an explicit multiplication and check by hand that all of the relevant DG axioms are satisfied. This can be done easily in the case of the Koszul complex. This is also the approach used to show that the Taylor resolution is a DG $R$-algebra. Indeed, this is the approach that will be employed in the proof of Theorem \ref{thm:BFSisDG}.

The purpose of the current paper is to expand on techniques introduced in \cite{mf} and explore applications related to the construction of DG structures on particular types of complexes. More precisely, we construct a DG structure on the length $4$ ``Big From Small Construction" (see Definition \ref{bfs}) introduced in \cite{bfs}. In the case where the characteristic of $R$ is not $2$, this was already proved in \cite{minres}, heavily relying on the construction of a ``complete higher order multiplication" as given by Palmer in \cite{algstruct}. 

In the present case, the DG structure is built from the ground up, and the construction is totally characteristic free. It is the intent of the author to illustrate the use of Tate-like complexes to construct interesting homomorphisms whose properties are often ideal for inducing DG structures on complexes. We construct the DG structure on the big from small construction since its structure is particularly well-suited to quotients by DG ideals; indeed, in the case where $R$ has characteristic $\neq 2$, it is shown in \cite{minres} that the minimal resolution of a grade $4$ almost complete intersection admits the structure of a commutative associative DG $R$-algebra by successively taking DG quotients of the big from small construction applied to the appropriate Gorenstein ideal. 

The paper is organized as follows: in Section \ref{dgalg}, basic facts on DG $R$-algebras with divided powers and exhibiting Poincar\'e duality are presented. The basic setup to be used throughout the paper, Setup \ref{setup2}, is also introduced. In Section \ref{bigsmall}, we give a brief review of the big from small construction and its relation to the notion of tight double linkage (see Definition \ref{tdl}) of Gorenstein ideals.

Section \ref{betamap} is where the aforementioned techniques are introduced and employed. To be precise, we construct a Tate-like complex $B$ along with a morphism of complexes $c : B \to K[-2]$, where $K$ is a length $3$ Koszul complex. By acyclicity, there is an induced chain homotopy $h$; moreover, this homotopy may be modified in such a way that certain homomorphisms induced by Poincar\'e duality will satisfy the required associativity relations for the contended DG structure on the big from small construction. In Section \ref{itsdg}, we prove that the length $4$ big from small construction admits the structure of a commutative associative DG $R$-algebra exhibiting Poincar\'e duality. As previously mentioned, this amounts to writing down an explicit multiplication and checking that all of the relevant DG axioms are satisfied. 

\section{Background on DG $R$-Algebras}\label{dgalg}

In this section, we introduce some necessary background, especially on DG $R$-algebras. We recall a result of Kustin on the algebra structures of self-dual resolutions of length $4$ (see Theorem \ref{l4resdg}). Finally, we conclude with the setup to be used for the rest of the paper. 

Throughout the paper, $R$ will denote a commutative Noetherian ring.

\begin{definition}\label{def:stdDefs}
The \emph{grade} of a proper ideal $I \subseteq R$ is the length of the longest regular sequence on $R$ in $I$. An ideal $I$ is \emph{perfect} if $\grade (I) = \pd_R (I)$ (the projective dimension). An ideal of grade $g$ is called \emph{Gorenstein} if it is perfect and $\ext_R^g (R/I , R) \cong R/I$.

A complex $F_\bullet : \cdots \to F_2 \to F_1 \to F_0 \to 0$ is called \emph{acyclic} if the only nonzero homology occurs at the $0$th position. A complex $F_\bullet$ is a \emph{free resolution} of $R/I$ if $H_0 (F_\bullet ) = R/I$ and all $F_i$ are free. 

Let $F$ denote a free $R$-module. The divided power algebra $D_\bullet (F)$ is the graded dual of the symmetric algebra $S_\bullet (F)$, which satisfies the following properties:
\begin{enumerate}[(a)]
    \item $x^{(0)} = 1$, $x^{(1)} = x$, and $x^{(k)} \in D_k F$ for $x \in F$.
    \item $x^{(p)} x^{(q)} = \binom{p+q}{q} x^{(p+q)}$ for $x \in F$,
    \item $(x+y)^{(p)} = \sum_{k=0}^p x^{(p-k)} y^{(k)}$ for $x$, $y \in F$. 
    \item $(r x)^{(p)} = r^p x^{(p)}$ for $r \in R$, $x \in F$.
    \item $(x^{(p)})^{(q)} = \frac{(pq)!}{q! (p^q)!} x^{(pq)}$ for $x \in F$.
\end{enumerate}
Observe that by the divided power algebra structure, given $x, \ x' \in F$, 
$$(x+x')^{(2)} = x^{(2)} + x \cdot x' +x'^{(2)}$$
whence it suffices to determine the action of a homomorphism with domain $D_2 (F)$ on elements of the form $x^{(2)}$. 

Let $A$, $B$, and $C$ be $R$-modules. A pairing $A \otimes_R B \to C$ is \emph{perfect} if the induced maps
$$A \to \hom_R (B , C) \quad \textrm{and} \quad B \to \hom_R (A , C)$$
are isomorphisms.
\end{definition}

\begin{definition}\label{dg}
A \emph{differential graded $R$-algebra} $(F,d)$ (DG $R$-algebra) over a commutative Noetherian ring $R$ is a complex of finitely generated free $R$-modules with differential $d$ and with a unitary, associative multiplication $F \otimes_R F \to F$ satisfying
\begin{enumerate}[(a)]
    \item $F_i F_j \subseteq F_{i+j}$,
    \item $d_{i+j} (x_i x_j) = d_i (x_i) x_j + (-1)^i x_i d_j (x_j)$,
    \item $x_i x_j = (-1)^{ij} x_j x_i$, and
    \item $x_i^2 = 0$ if $i$ is odd,
\end{enumerate}
where $x_k \in F_k$. A DG $R$-algebra $F$ is a DG$\Gamma$ $R$-algebra if for each positive even index $i$ and each $x_i \in F_i$, there is a family of elements $\{ x_i^{(k)} \}$ satisfying the divided power axioms (see Definition \ref{def:stdDefs}) in the module $D_k F_i$, along with the extra condition $d_{ik}(x_i^{(k)}) = d_i (x_i) x_i^{(k-1)}$.

A DG $R$-algebra $F$ exhibits \emph{Poincar\'e duality} if there is an integer $m$ such that $F_i =0$ for $i>m$, $F_m \cong R$, and for each $i$ the multiplication map
$$F_i \otimes_R F_{m-i} \to F_m$$
is a perfect pairing of $R$-modules. Given a DG $R$-algebra $F$ such that $F_i =0$ for $i>m$, an orientation isomorphism is a choice of isomorphism
$$[ - ]_F : F_m \to R,$$
given that such an isomorphism exists.
\end{definition}

\begin{remark}
Let $F$ be a DG $R$-algebra exhibiting Poincar\'e duality with $m$ such that $F_i = 0$ for $i >m$. Given any $R$-module $M$, observe that in order to specify an $R$-module homomorphism $M \to F_i$ for any $i \leq m$, it suffices to construct a morphism $M \otimes_R F_{m-i} \to F_m$. Any such map induces a morphism $M \to \hom_R (F_{m-i} , F_m)$, and by Poincar\'e duality, $\hom_R (F_{m-i} , F_m) \cong F_i$. 
\end{remark}

\begin{example}
Let $F$ be a free $R$-module of rank $n$ and $\psi : F \to R$ an $R$-module homomorphism. The Koszul complex $(K,k)$ is the complex with $K_i := \bigwedge^i F$ and differential $k_i : K_i \to K_{i-1}$ defined as the composition
\begingroup\allowdisplaybreaks
\begin{align*}
    \bigwedge^i F &\xrightarrow{\Delta} F \otimes \bigwedge^{i-1} F \\
    &\xrightarrow{\psi \otimes 1} R \otimes \bigwedge^{i-1} F \cong \bigwedge^{i-1} F \quad (\Delta \ \textrm{denotes comultiplication}).
\end{align*}
\endgroup
Then the Koszul complex is a DG$\Gamma$ $R$-algebra exhibiting Poincar\'e duality with product defined by exterior multiplication. 
\end{example}

\begin{lemma}\label{htpmod}
Let $c : (B , b) \to (A , a)$ be a morphism of complexes of projective modules with $(A,a)$ acyclic. Assume $B_i = B_i' \oplus B_i''$, $b_i (B_i') \subseteq B_{i-1}'$ and $c_i (B_i') = 0$ for each $i$, with $B_0 = B_0'$. Then there exists a homotopy $h : B \to A[1]$ such that $h_i (B_i') = 0$ for each $i$.
\end{lemma}

\begin{proof}
Build the sequence of homotopies inductively, setting $h_0 = 0$. Let $i>0$ and assume $h_{i-1}$ has been defined. Notice that $a_i \circ (c_i - h_{i-1} \circ b_i)=0$ by the inductive hypothesis. By acyclicity of $A$, there exists $h_i : B_i \to A_{i+1}$ such that 
$$a_{i+1} \circ h_i = c_i - h_{i-1} \circ b_i.$$
Moreover, $(c_i - h_{i-1} \circ b_i)(B_i') = 0$ by the assumptions on $c$ and $b$, and induction. Thus $h_i$ may be chosen to satisfy the desired properties.
\end{proof}

\begin{theorem}[{\cite[Theorem 4.6]{l4res}}]\label{l4resdg}
Let $R$ be a commutative Noetherian ring and $M$ a length $4$ resolution of a cyclic $R$-module by finitely generated free $R$-modules with $\rank (M_4 ) =1$ and
$$M \cong \hom_R (M , M_4).$$
Then $M$ has the structure of a DG$\Gamma$ $R$-algebra which exhibits Poincar\'e duality.
\end{theorem}

\begin{remark}
In particular, Theorem \ref{l4resdg} tells us that in the case that $R$ is local or standard graded, the (homogeneous) minimal free resolution of a grade $4$ Gorenstein ideal admits the structure of an associative DG$\Gamma$ $R$-algebra exhibiting Poincar\'e duality.
\end{remark}

The following setup will be used for the rest of the paper. Notice that $R$ need not be local.

\begin{setup}\label{setup2}
Let $I$ be a grade $4$ Gorenstein ideal. Let $(M,m) : 0 \to M_4 \to \cdots \to M_1 \to M_0 = R$ be a length $4$ resolution of $R/I$. Assume that $M$ is also a DG$\Gamma$ $R$-algebra exhibiting Poincar\'e duality.

Let $(K,k)$ denote the Koszul complex on a length $3$ regular sequence. Let $R/\mfa$ denote the complete intersection that $K$ is resolving; assume $\mfa \subseteq I$. Moreover, assume that $M_1 = M_{1,1} \oplus M_{1,2}$, where $\rank (M_{1,1}) = 3$ and $m_1(M_{1,1}) = \mfa$. 

Define $\alpha_0 : K_0 = R \to M_0 = R$ as the identity. Define $\alpha_1 : K_1 \to M_1$ by the condition that the composition $K_1 \to M_1 \to M_{1,1}$ makes the following diagram commute:
$$\xymatrix{K_1 \ar[r]^-{k_1} \ar[d] & \mfa \\
 M_{1,1} \ar[ur]_{m_1} & \\}$$
and that $K_1 \to M_1 \to M_{1,2}$ is the $0$ map. Define $\alpha : K \to M$ as the map which extends $\alpha_1$ and $\alpha_0$ as described above. Given this data, define $\beta : M \to K[-1]$ via
$$\big[ \beta_i (\theta_i ) \cdot \phi_{4-i} \big]_{K} =  (-1)^{i+1} \big[ \theta_i \alpha_{4-i} (\phi_{4-i}) \big]_{M}$$
\end{setup}

\begin{remark}
The assumption that $\alpha_1 (K_1)$ is a direct summand of $M_1$ means that the resolution of $M$ is not necessarily minimal in the case that $R$ is local. For our purposes, this will not be significant since Proposition \ref{bfstdl} applies regardless of the minimality of $M$.
\end{remark}

\section{The Big From Small Construction}\label{bigsmall}

In this section we give a brief review of the big from small construction of \cite{bfs} as applied to the data Setup \ref{setup2}. For convenience, we recall the relationship between the complex of Definition \ref{bfs} and \emph{tightly double linked} Gorenstein ideals. Proposition \ref{bfstdl} is stated here for reference and will not be directly employed in later sections.

\begin{definition}\label{bfs}
Adopt notation and hypotheses as in Setup \ref{setup2}; let $r \in R$. The \emph{big from small construction} applied to the data $(\alpha , r)$ yields the complex
$$(F ( \alpha , r), f): 0 \to F_4 \to F_3 \to F_2 \to F_1 \to F_0$$
with $F_0 = R$, $F_1 = K_1 \oplus M_1$, $F_2 = K_2 \oplus M_2 \oplus K_1$, $F_3 = M_3 \oplus K_2$, $F_4 = M_4$, and
\begingroup
\allowdisplaybreaks
\begin{align*}
    f_1 &= (m_1 \ \beta_1 + r k_1 ) \\
    f_2 = \begin{pmatrix}
    m_2 & \beta_2 & r \\
    0 & - k_2 & - \alpha_1 \\
    \end{pmatrix}, &\qquad \qquad f_3 = \begin{pmatrix} \beta_3 & -r \\
    - k_3 & - \alpha_2 \\
    0 & m_2 \\
    \end{pmatrix} \\
    f_4 &= \begin{pmatrix} \alpha_3 \beta_4 -r m_4 \\
    -k_3 \beta_4 \\
    \end{pmatrix} \\
\end{align*}
\endgroup
\end{definition}

\begin{definition}\label{tdl}
Let $I$ and $I'$ be two grade $g$ Gorenstein ideals in a commutative Noetherian ring $R$. Suppose $\mfa$ is a grade $g-1$ complete intersection. There is a \emph{tight double link} between $I$ and $I'$ over $K$ if there exists an ACI $J = (\mfa , y , y')$ with $(\mfa , y)$ and $(\mfa , y')$ both complete intersections, and
$$(\mfa , y ): I = (\mfa, y' ) : I$$
\end{definition}
The relationship between Definitions \ref{bfs} and \ref{tdl} is given by the following Proposition:

\begin{prop}[{\cite[Proposition 4.3]{minres}}]\label{bfstdl}
Let $\mfa \subseteq I$ be ideals in a commutative Noetherian ring $R$ with $\mfa$ a grade $3$ complete intersection and $I$ a grade $4$ Gorenstein ideal. Let $K$ be a length $3$ Koszul complex resolving $R/\mfa$ and $M$ a length $4$ Poincar\'e DG $R$-algebra resolving $R/I$, with $\alpha : K \to M$ the induced comparison map.

If there is a tight double link between $I$ and a grade $4$ Gorenstein ideal $I'$ over $\mfa$, then there exists $r \in R$ such that $F ( \alpha , r)$ resolves $R/I'$.
\end{prop}

\section{Using Tate-Like Complexes To Induce DG Structures}\label{betamap}

In this section, we develop the machinery and construct the pieces to be used in defining the multiplication on the complex of Definition \ref{bfs}. Many of the computations of this section are straightforward yet tedious verifications of desired properties. A brief overview of this section is as follows: first, Proposition \ref{morcx} provides a morphism of complexes from a Tate-like complex (see Definition \ref{def:tateDiff}) to the complex $M$. This morphism of complexes will be nullhomotopic by some nullhomotopy $h$, and one can furthermore employ Lemma \ref{htpmod} to ensure that $h$ may be chosen to satisfy an additional list of properties (see Corollary \ref{h1}). Finally, we use the Poincar\'e duality exhibited by $M$ to define two additional maps $X$ and $X^t$; Proposition \ref{properties} enumerates the properties of $X$ and $X^t$ that will be needed in the proof of Theorem \ref{thm:BFSisDG}. 

We begin this section by examining additional properties of the maps $\alpha$ and $\beta$ as in Setup \ref{setup2}.

\begin{prop}
Adopt notation and hypotheses as in Setup \ref{setup2}. Then:
\begin{enumerate}
    \item $\beta \circ \alpha = 0$,
    \item $\beta_{i+j} (\alpha_i (\phi_i ) \theta_j) =  \phi_i \beta_j (\theta_j)$, and
    \item $\beta_4$ is an isomorphism.
\end{enumerate}
\end{prop}

\begin{proof}
$(1)$: Let $\phi_i \in K_i$, $\phi_{4-i} \in K_{4-i}$. Then,
\begingroup\allowdisplaybreaks
\begin{align*}
    \big[ \beta_i (\alpha_i (\phi_i)) \phi_{4-i} \big]_K &= (-1)^{i+1} \big[ \alpha_i ( \phi_i ) \alpha_{4-i} (\phi_{4-i} ) \big]_M \\
    &= (-1)^{i+1} \big[ \alpha_4 (\phi_i \w \phi_{4-i} ) \big]_M \\
    &= 0, \\ 
\end{align*}
\endgroup
since $K_i = 0$ for $i >3$. \\
$(2)$: Let $\phi_{4-i-j} \in K_{4-i-j}$.
\begin{equation*}
    \begin{split}
        \big[ \beta_{i+j} \big( \theta_j \cdot \alpha_i (\phi_i ) \big) \cdot \phi_{4-i-j} \big]_K &= (-1)^{i+j+1} \big[ \theta_j \cdot \alpha_i (\phi_i ) \cdot \alpha_{4-i-j} (\phi_{4-i-j} ) \big]_M \\
        &= (-1)^{i+j+1} \big[ \theta_j \cdot \alpha_{4-j} (\phi_i \w \phi_{4-i-j} ) \big]_M \\
        &= (-1)^i \big[ \big( \beta_j (\theta_j ) \cdot \phi_i \big) \cdot \phi_{4-i-j} \big]_K, \\
    \end{split}
\end{equation*}
so that $\beta_{i+j} (\theta_j \alpha_i (\phi_i) ) = (-1)^i \beta_j( \theta_j ) \phi_i$. Using skew-commutativity, the result follows. \\
$(3)$: Notice that $\alpha_0$ is the identity map, which is an isomorphism. By Poincar\'e duality, $\beta_4$ must also be an isomorphism.
\end{proof}

\begin{definition}\label{def:tateDiff}
Let $(F_\bullet, d_\bullet)$ be a complex of length $n$ concentrated in nonnegative homological degrees. Given an integer $a \geq 0$, let 
$$N_a (F_i ) := \begin{cases} \bigwedge^a F_i & \textrm{if} \ i \ \textrm{is odd}, \\
D_a (F_i) & \textrm{if} \ i \ \textrm{is even}. \\
\end{cases}$$
There is an induced map $\partial_{a,i} : N_a (F_i) \to N_{a-1} (F_i)$ defined as the composition
\begingroup\allowdisplaybreaks
\begin{align*}
    N_a (F_i) & \xrightarrow{\Delta} F_i \otimes N_{a-1} (F_i) \\
    & \xrightarrow{d_i \otimes 1} F_{i-1} \otimes N_{a-1} (F_i), 
\end{align*}
\endgroup
where $\Delta$ denotes the appropriate comultiplication. Each $\partial_{a,i}$ induces maps
$$\partial_k^j : \bigoplus_{\substack{a_1 + \cdots + n a_n = k\\ 
a_0 + \cdots + a_n = j \\}} N_{a_0} (F_0) \otimes \cdots \otimes N_{a_n} (F_n) \to \bigoplus_{\substack{a_1 + \cdots + na_n = k-1 \\
a_0 + \cdots + a_n = j \\}} N_{a_0} (F_0) \otimes \cdots \otimes N_{a_n} (F_n)$$
by employing a graded Leibniz rule on each tensor. The map $\partial_k^j$ will be referred to as \emph{the induced Tate differential}, and the complex induced by the collection of all $\partial_k^j$ will be referred to as a \emph{Tate-like complex}.
\end{definition}
In the following Proposition, the complex $B$ is an example of a Tate-like complex. 

\begin{prop}\label{morcx}
Adopt notation and hypotheses as in Setup \ref{setup2}. Let $$B: B_5 \to B_4 \to B_3 \to B_2 \to B_1 \to B_0$$ be the complex with
$$B_0 = M_0, \quad B_1 = M_1 ,\quad B_2 = (\bigwedge^2 M_1) \oplus M_2$$
$$B_3 = \bigwedge^3 M_1 \oplus (M_1 \otimes M_2) \oplus M_3, \quad B_4 = (\bigwedge^2 M_1 \otimes M_2) \oplus D_2 M_2 \oplus (M_1 \otimes M_3)$$
$$B_5 = (M_1 \otimes D_2 M_2) \oplus( \bigwedge^2 M_1 \otimes M_3)$$
and the differential $d$ being the induced Tate differential (that is, just use the graded Leibniz rule). Define a map $c : B \to K[-2]$ via:
$$c_0 = c_1 = c_2 = 0,$$
\begingroup
\allowdisplaybreaks
\begin{align*}
        c_3 \Bigg( \begin{pmatrix} \theta_1 \w \theta_1'\w \theta_1'' \\
        \theta_1''' \otimes \theta_2 \\
        \theta_3 \\ \end{pmatrix} 
        \Bigg) &= \beta_2 (\theta_1 \w \theta_1' ) \beta_1 (\theta_1'') - \beta_2 (\theta_1 \w \theta_1'' ) \beta_1 (\theta_1') \\
        & \qquad+ \beta_2 (\theta_1' \w \theta_1'' ) \beta_1 (\theta_1), \\
\end{align*}
\endgroup
\begingroup
\allowdisplaybreaks
\begin{align*}
        c_4 \Bigg( \begin{pmatrix} \theta_1 \w \theta_1' \otimes \theta_2' \\
         \theta_2^{(2)} \\
        \theta_1''' \otimes \theta_3 \\ \end{pmatrix} 
        \Bigg) &= \beta_1 (\theta_1') \beta_3 (\theta_1 \theta_2') -\beta_1 ( \theta_1) \beta_3 (\theta_1'\theta_2' ) \\
        &\qquad - \beta_2 (\theta_1 \w \theta_1' ) \beta_2 (\theta_2'), \\
    \end{align*}
\endgroup
\begingroup
\allowdisplaybreaks
\begin{align*}
        c_5 \Bigg( \begin{pmatrix} \theta_1 \otimes \theta_2^{(2)} \\
        \theta_1' \w \theta_1'' \otimes \theta_3 \\ \end{pmatrix} 
        \Bigg) &= \beta_1 (\theta_1 ) \beta_4 (\theta_2^{(2)}) - \beta_3 (\theta_1 \theta_2 ) \beta_2 (\theta_2) \\
        &\qquad -\beta_1 (\theta_1'' ) \beta_4 ( \theta_1' \theta_3 ) + \beta_1 (\theta_1' ) \beta_4 (\theta_1'' \theta_3 ) \\
        &\qquad- \beta_2 (\theta_1' \w \theta_1'' ) \beta_3 (\theta_3). 
    \end{align*}
\endgroup
Then $c$ is a morphism of complexes.
\end{prop}

\begin{proof}
One must verify that all of the appropriate maps commute. The first nontrivial place to check is $c_3 : B_3 \to K_1$. 
\begingroup
\allowdisplaybreaks
\begin{align*}
        k_1 \circ c_3 \Bigg( \begin{pmatrix} \theta_1 \w \theta_1'\w \theta_1'' \\
        \theta_1''' \otimes \theta_2 \\
        \theta_3 \\ \end{pmatrix} 
        \Bigg) &= \beta_1 ( m_2 (\theta_1 \w \theta_1') ) \beta_1 (\theta_1'') - \beta_1 (m_2 (\theta_1 \w \theta_1'') ) \beta_1 (\theta_1') \\
        & \qquad+ \beta_1 (m_2 (\theta_1' \w \theta_1'') ) \beta_1 (\theta_1) \\
        &= \beta_1 ( m_1 (\theta_1) \theta_1' ) \beta_1 (\theta_1'')-  \beta_1 ( \theta_1 m_1( \theta_1') ) \beta_1 (\theta_1'') \\
        &\qquad - \beta_1 (m_1 (\theta_1) \theta_1'' ) \beta_1 (\theta_1') + \beta_1 ( \theta_1 m_1( \theta_1'') ) \beta_1 (\theta_1') \\
        &\qquad + \beta_1 (m_1 (\theta_1' ) \theta_1'' ) \beta_1 (\theta_1) - \beta_1 ( \theta_1' m_1( \theta_1'') ) \beta_1 (\theta_1). 
    \end{align*}
\endgroup
Since $\beta_i$ is a map of $R$-modules, one finds that $\beta_1 ( m_1 (\theta_1) \theta_1' ) \beta_1 (\theta_1'') = \beta_1 (  \theta_1' ) \beta_1 (m_1(\theta_1) \theta_1'')$, so this term cancels with the first term of the second line. Similarly for the other terms, so this composition is $0$ as needed. Next:
\begingroup
\allowdisplaybreaks
\begin{align*}
        k_2 \circ c_4 \Bigg( \begin{pmatrix} \theta_1 \w \theta_1' \otimes \theta_2' \\
         \theta_2^{(2)} \\
        \theta_1''' \otimes \theta_3 \\ \end{pmatrix} 
        \Bigg) &= \beta_1 (\theta_1') \beta_2 (m_3(\theta_1 \theta_2')) -\beta_1 ( \theta_1) \beta_2 (m_3(\theta_1'\theta_2') ) \\
        &\qquad - \beta_1 (m_2(\theta_1 \w \theta_1') ) \beta_2 (\theta_2')+\beta_2 (\theta_1 \w \theta_1' ) \beta_1 (m_2(\theta_2')) \\
        &= \beta_1 (\theta_1') \beta_2 (m_1 (\theta_1) \theta_2') - \beta_1 (\theta_1') \beta_2 (\theta_1 m_2( \theta_2')) \\
        &\qquad -\beta_1 ( \theta_1) \beta_2 (m_1(\theta_1' )\theta_2' ) + \beta_1 ( \theta_1) \beta_2 (\theta_1'm_2( \theta_2') ) \\
        &\qquad - \beta_1 (m_1(\theta_1) \w \theta_1' ) \beta_2 (\theta_2') + \beta_1 (\theta_1 \w m_1(\theta_1') ) \beta_2 (\theta_2') \\
        &\qquad \beta_2 (\theta_1 \w \theta_1' ) \beta_1 (m_2(\theta_2')) \\
        &= \beta_2 ( \theta_1 \w \theta_1 ' ) \beta_1 (m_2 (\theta_2' ) )  + \beta_1 (\theta_1) \beta_2 (\theta_1' \w m_2 (\theta_2') \\
        &\qquad -\beta_1 (\theta_1') \beta_2 (\theta_1' \w m_2 (\theta_2') ) \\
        &= c_3 \Bigg( \begin{pmatrix} \theta_1 \w \theta_1' \w m_2(\theta_2') \\
        m_2(\theta_2) \otimes \theta_2 + m_2 (\theta_1 \w \theta_1') \otimes \theta_2' -\theta_1'' \otimes m_3(\theta_3) \\
        m_1 (\theta_1'') \theta_3  \\
        \end{pmatrix} \Bigg) \\
        &= c_3 \circ d \Bigg( \begin{pmatrix} \theta_1 \w \theta_1' \otimes \theta_2' \\
         \theta_2^{(2)} \\
        \theta_1''' \otimes \theta_3 \\ \end{pmatrix} 
        \Bigg). 
    \end{align*}
\endgroup
And, finally:
\begingroup
\allowdisplaybreaks
\begin{align*}
        k_3 \circ c_5 \Bigg( \begin{pmatrix} \theta_1 \otimes \theta_2^{(2)} \\
        \theta_1' \ \theta_1'' \otimes \theta_3 \\ \end{pmatrix} 
        \Bigg) &= \beta_1 (\theta_1 ) \beta_3 (m_4 (\theta_2^{(2)})) - \beta_2 (m_3(\theta_1 \theta_2 ) ) \beta_2 (\theta_2) \\
        &\qquad -\beta_3 (\theta_1 \theta_2 ) \beta_1 (m_2 (\theta_2) ) \\
        &\qquad -\beta_1 (\theta_1'' ) \beta_3 ( m_4(\theta_1' \theta_3 ) ) + \beta_1 (\theta_1' ) \beta_3 (m_4(\theta_1'' \theta_3 ) ) \\
        &\qquad- \beta_1 ( m_2(\theta_1' \w \theta_1'') ) \beta_3 (\theta_3) + \beta_2 (\theta_1' \w \theta_1'' ) \beta_2 (m_3 (\theta_3 ) ) \\
        &= \beta_1 (\theta_1 ) \beta_3 (m_2 (\theta_2) \theta_2) \\
        &\qquad - \beta_2 (m_1(\theta_1) \theta_2  ) \beta_2 (\theta_2) + \beta_2 (\theta_1 m_2(\theta_2 ) ) \beta_2 (\theta_2) \\
        &\qquad -\beta_3 (\theta_1 \theta_2 ) \beta_1 (m_2 (\theta_2) ) \\
        &\qquad -\beta_1 (\theta_1'' ) \beta_3 ( m_1(\theta_1') \theta_3 ) + \beta_1 (\theta_1'' ) \beta_3 ( \theta_1' m_3( \theta_3 ) ) \\
        &\qquad +  \beta_1 (\theta_1' ) \beta_3 (m_1(\theta_1'') \theta_3  ) -  \beta_1 (\theta_1' ) \beta_3 (\theta_1'' m_3( \theta_3 ) ) \\
        &\qquad +\beta_1 ( m_1(\theta_1')  \theta_1'' ) \beta_3 (\theta_3) - \beta_1 ( \theta_1' m_1( \theta_1'') ) \beta_3 (\theta_3) \\
        &\qquad +\beta_2 (\theta_1' \w \theta_1'' ) \beta_2 (m_3 (\theta_3 ) ). 
     \end{align*}
\endgroup
Notice that $\beta_2 (m_1(\theta_1) \theta_2  ) \beta_2 (\theta_2) = m_1(\theta_1) \beta_2 (\theta_2) \beta_2 (\theta_2) = 0$, since $\beta_2 (\theta_2) \in K_1$. Moreover, $\beta_1 (\theta_1'' ) \beta_3 ( m_1(\theta_1') \theta_3 ) = \beta_1 ( m_1 (\theta_1') \theta_1'' ) \beta_3 (  \theta_3 )$, so this term cancels with the first term on the second line from the bottom. The same holds for $\beta_1 ( m_1(\theta_1')  \theta_1'' ) \beta_3 (\theta_3)$. Thus one is left with:
\begingroup
\allowdisplaybreaks
\begin{align*}
        &= \beta_1 (\theta_1 ) \beta_3 (m_2 (\theta_2) \theta_2) + \beta_2 (\theta_1 m_2(\theta_2 ) ) \beta_2 (\theta_2) \\
        &\qquad - \beta_3 (\theta_1 \theta_2 ) \beta_1 (m_2 (\theta_2) ) +\beta_2 (\theta_1' \w \theta_1'' ) \beta_2 (m_3 (\theta_3 ) ) \\
        &\qquad -\beta_1 (\theta_1' ) \beta_3 (\theta_1'' m_3( \theta_3 ) ) +\beta_1 (\theta_1'' ) \beta_3 ( \theta_1' m_3( \theta_3 ) ) \\
        &= c_4 \Bigg( \begin{pmatrix} -\theta_1 \w m_2 (\theta_2) \otimes \theta_2 + \theta_1' \w \theta_1'' \otimes m_3(\theta_3)  \\
        m_1 (\theta_1) \theta_2^{(2)} \\
        m_2(\theta_1' \w \theta_1'') \otimes \theta_3 \\ \end{pmatrix} 
        \Bigg) \\
        &= c_4 \circ d  \Bigg( \begin{pmatrix} \theta_1 \otimes \theta_2^{(2)} \\
        \theta_1' \ \theta_1'' \otimes \theta_3 \\ \end{pmatrix} 
        \Bigg). 
    \end{align*}
\endgroup
This shows that the map $c$ is a morphism of complexes.
\end{proof}

\begin{prop}\label{cprop}
Adopt notation and hypotheses as in Proposition \ref{morcx}. Then the map $c_4$ satisfies:
$$c_4 \Big( M_1 \w \alpha_1 (K_1) \otimes M_2 \Big) = 0,$$
$$c_4 \Big( \bigwedge^2 M_1 \otimes  \alpha_1(K_2) \Big) = 0,$$
$$c_4 \big( \theta_1 \w \theta_1' \otimes \alpha_1 (\phi_1)\cdot \theta_1 \big) = 0,$$
$$c_4 \Big( \theta_1 \w \theta_1' \otimes \alpha_1 (\phi_1) \theta_1'' + \theta_1'' \w \theta_1' \otimes \alpha_1 (\phi_1) \theta_1 \Big) = 0.$$
Additionally, the map $c_3$ satisfies:
$$c_3 ( \bigwedge^2 M_1 \w \alpha_1 (K_1) ) = 0.$$
\end{prop}

\begin{proof}
For the first formula,
\begingroup
\allowdisplaybreaks
\begin{align*}
        c_4 (\theta_1 \w \alpha_1( \phi_1 ) \otimes \theta_2 ) &= \beta_1 (\alpha_1 (\phi_1)) \beta_3 (\theta_1 \theta_2) -\beta_1 ( \theta_1) \beta_3 (\alpha_1(\phi_1) \theta_2 ) \\
        &\qquad - \beta_2 (\theta_1 \w \alpha_1 ( \phi_1) ) \beta_2 (\theta_2). 
   \end{align*}
   \endgroup
Since $\beta \circ \alpha = 0$, the first term vanishes. For the second terms,
$$\beta_3 ( \alpha_1 (\phi_1) \theta_2 ) = \phi_1 \beta_2 (\theta_2),$$
$$\beta_2 (\theta_1 \w \alpha_1 (\phi_1) ) = -\beta_1 (\theta_1 ) \phi_1,$$
so that these terms cancel and we obtain $0$. In the next case,
\begingroup
\allowdisplaybreaks
\begin{align*}
        c_4 (\theta_1 \w \theta_1' \otimes \alpha_2 (\phi_2) ) &= \beta_1 ( \theta_1') \beta_3 (\theta_1 \alpha_2 (\phi_2)) -\beta_1 ( \theta_1) \beta_3 ( \theta_1' \alpha_2 (\phi_2) ) \\
        &\qquad - \beta_2 (\theta_1 \w  \theta_1' ) \beta_2 (\alpha_2 (\phi_2)). 
    \end{align*}
   \endgroup
In this case, the last term is $0$. For the first two terms,
$$\beta_3 (\theta_1 \alpha_2 (\phi_2)) = \beta_1 (\theta_1 ) \phi_2,$$
$$\beta_3 ( \theta_1' \alpha_2 (\phi_2) ) = \beta_1 (\theta_1') \phi_2,$$
so these terms again cancel. For the last property of $c_4$:
\begingroup
\allowdisplaybreaks
\begin{align*}
        c_4 \Big( \theta_1 \w \theta_1' \otimes \alpha_1 (\phi_1) \theta_1'' \Big) &= \beta_1 ( \theta_1') \beta_3 (\theta_1 \alpha_1 (\phi_1) \theta_1'') -\beta_1 ( \theta_1) \beta_3 ( \theta_1' \alpha_1 (\phi_1) \theta_1'' ) \\
        &\qquad - \beta_2 (\theta_1 \w  \theta_1' ) \beta_2 (\alpha_1 (\phi_1) \theta_1'') \\
        &= -\beta_1 ( \theta_1') \beta_3 (\theta_1'' \alpha_1 (\phi_1) \theta_1) +\beta_1 ( \theta_1) \phi_1 \beta_2 (\theta_1' \theta_1'' ) \\
        &\qquad - \beta_2 (\theta_1 \w  \theta_1' ) \phi_1 \beta_1 ( \theta_1'') \\
        &= - \beta_1 ( \theta_1') \beta_3 (\theta_1'' \alpha_1 (\phi_1) \theta_1) + \beta_2 (\theta_1'' \w \theta_1' ) \beta_2 (\alpha_1 (\phi_1) \theta_1 ) \\
        &\qquad + \beta_1 (\theta_1'') \beta_3 (\theta_1 \alpha_1 (\phi_1 ) \theta_1') \\
        &= -c_4 \Big( \theta_1'' \w \theta_1' \otimes \alpha_1 (\phi_1) \theta_1 \Big).
    \end{align*}
   \endgroup
 Moreover, setting $\theta_1 = \theta_1''$ in the above, notice that the second equality in the above computation becomes $0$, so the penultimate property for $c_4$ also holds.
 
 For the $c_3$ property, one computes in similar fashion:
 \begingroup
 \allowdisplaybreaks
 \begin{align*}
     c_3( \theta_1 \w \theta_1 ' \w \alpha_1 ( \phi_1) ) &= \beta_2 ( \theta_1 \w \theta_1 ' ) \beta_1 ( \alpha_1 ( \phi_1)) - \beta_2 ( \theta_1 \w \alpha_1 ( \phi_1) ) \beta_1 ( \theta_1') \\
     &+ \beta_2 ( \theta_1' \w \alpha_1 ( \phi_1) ) \beta_1 ( \theta_1) \\
     &=-\beta_1 ( \theta_1) \phi_1 \beta_1 ( \theta_1') + \beta_1 ( \theta_1') \phi_1 \beta_1 ( \theta_1) \\
     &=0,
 \end{align*}
 \endgroup
 where in the above, we have used that $\beta \circ \alpha = 0$ and $\beta_{i+j} (\alpha_i (\phi_i ) \theta_j) =  \phi_i \beta_j (\theta_j)$. 
\end{proof}
The following Corollary provides the modified nullhomotopy used to define the maps $X$ and $X^t$ of Definition \ref{Xnew}.

\begin{cor}\label{h1}
Adopt notation and hypotheses as in Proposition \ref{cprop}. Then there exists a homotopy $h: B \to K[-1]$ with $c = kh + h d$. Moreover, $h$ may be chosen to satisfy the following:
\begin{enumerate}
    \item $h$ restricted to any summand of each $B_i$ with fewer than $3$ terms in the product is identically $0$.
    \item $h_3 \Big( \bigwedge^2 M_1 \w \alpha_1 (K_1) ) = 0$,
    \item $h_4 \Big( M_1 \w \alpha_1 (K_1) \otimes M_2 \Big) = 0$,
    \item $h_4 \Big( \bigwedge^2 M_1 \otimes  \alpha_1(K_2) \Big) = 0$,
    \item $h_4 (\theta_1 \w \theta_1' \otimes \alpha_1 (\phi_1)\cdot \theta_1) = 0$,
    \item $h_4 \Big( \theta_1 \w \theta_1' \otimes \alpha_1 (\phi_1) \theta_1'' + \theta_1'' \w \theta_1' \otimes \alpha_1 (\phi_1) \theta_1 \Big) = 0$.
\end{enumerate}
\end{cor}

\begin{proof}
The existence of the homotopy follows from the fact that $c$ is a morphism of complexes with $c_0 = 0$ and $K$ is acyclic. The fact that one may arrange $h_4$ to have property $(1)$ follows from the definition of $c$ and Lemma \ref{htpmod}. It is clear that $\bigwedge^2 M_1 \w \alpha_1 (K_1)$ is a direct summand of $B_3$ and $M_1 \w \alpha_1 (K_1) \otimes M_2$ is a direct summand of $B_4$ by the splitting assumption $M_1 = \alpha_1(K_1) \oplus M_{1,2}$. This yields properties $(1)$, $(2)$, and $(3)$. Assume now that $h$ has been chosen to satisfy these properties.

For property $(4)$, applying the Tate differential yields:
\begingroup\allowdisplaybreaks
\begin{align*}
    d( \theta_1 \w \theta_1' \otimes \alpha_2 ( \phi_2) ) &= m_1 ( \theta_1 ) \theta_1 ' \otimes \alpha_2 ( \phi_2) - m_1 ( \theta_1') \theta_1 \otimes \alpha_2 ( \phi_2) \\
    &+\theta_1 \w \theta_1' \otimes \alpha_1 ( k_2 ( \phi_2) ) \\
    & \in \Big( M_1 \otimes M_2 \Big) \oplus \bigwedge^2 M_1 \w \alpha_1 (K_1). 
\end{align*}
\endgroup
By our selection of $h$, one has that $h_3 ( d( \theta_1 \w \theta_1' \otimes \alpha_2 ( \phi_2) ) = 0$. Since $c = hd + kh$, Proposition \ref{cprop} combined with the previous sentence yields
$$k_3 ( h_4 ( \theta_1 \w \theta_1' \otimes \alpha_2 ( \phi_2) ) ) = 0,$$
and since $k_3$ is injective, property $(4)$ follows.

For property $(5)$, one applies the Tate differential:
\begingroup\allowdisplaybreaks
\begin{align*}
    d ( \theta_1 \w \theta_1' \otimes \alpha_1 (\phi_1)\cdot \theta_1 ) &= m_1 ( \theta_1 ) \theta_1' \otimes \alpha_1 ( \phi_1) \theta_1 -m_1 ( \theta_1' ) \theta_1 \otimes \alpha_1 ( \phi_1) \theta_1 \\
    &+k_1 ( \phi_1) \theta_1 \w \theta_1' \w \theta_1 -  m_1( \theta_1 ) \theta_1 \w \theta_1' \w \alpha_1 ( \phi_1) \\
    &\in \Big( M_1 \otimes M_2 \Big) \oplus \bigwedge^2 M_1 \w \alpha_1 (K_1), 
\end{align*}
\endgroup
so that in an identical manner to property $(4)$, property $(5)$ follows. Finally, for property $(6)$, assume that $h$ has been chosen to satisfy property $(5)$ as well. Simply let $\theta_1 \mapsto \theta_1 + \theta_1'$ in $(5)$ to obtain $(6)$.
\end{proof}

\begin{definition}\label{Xnew}
Adopt notation and hypotheses as in Corollary \ref{h1}. Define $h_4 : \bigwedge^2 M_1 \otimes M_2 \to K_3$ by composing with the inclusion $\bigwedge^2 M_1 \otimes M_2 \to B_4$. Then, define $X : \bigwedge^2 M_1 \to M_2$, $X^t : M_1 \otimes M_2 \to M_3$ via
$$X(\theta_1 \w \theta_1' ) \cdot \theta_2 = (\beta_4^{-1} \circ h_4 ) (\theta_1 \w \theta_1' \otimes \theta_2 ) = \theta_1' \cdot X^t (\theta_1 \otimes \theta_2).$$
\end{definition}
Proposition \ref{properties} provides a list of properties that will be needed to show that the algebra structure of Theorem \ref{thm:BFSisDG} is associative and satisfies the graded Leibniz rule. The verification of these properties is tedious but necessary, and all details will be provided.

\begin{prop}\label{properties}
The maps $X$ and $X^t$ of Definition \ref{Xnew} have the following properties:
\begin{enumerate}
    \item $\beta_2 X(\theta_1 \w \theta_1' ) = 0,$
    \item $\beta_3 X^t (\theta_1 \otimes \theta_2 ) = 0,$
    \item $X^t (\theta_1 \otimes \alpha_2 ( \phi_2 ) )  =0$ and $\alpha_1 ( \phi_1)  \cdot X( \theta_1 \w \theta_1') = 0,$
    \item $\alpha_1 (\phi_1) \theta_1'' \cdot X (\theta_1 \w \theta_1') +\alpha_1 (\phi_1) \theta_1 \cdot X (\theta_1'' \w \theta_1') =0$. In particular, this implies $ \theta_1'' \cdot X (\theta_1 \w \theta_1') + \theta_1 \cdot X (\theta_1'' \w \theta_1') =0,$
    \item $m_2  X (\theta_1 \w \theta_1') = \beta_1 (\theta_1') \theta_1 - \beta_1 (\theta_1) \theta_1' - \alpha_1 \beta_2 (\theta_1 \theta_1'),$
    \item $X^t ( \theta_1 \otimes  m_3 (\theta_3) ) = \theta_1 \alpha_2 \beta_3 (\theta_3) - \beta_1 (\theta_1) \theta_3 - \alpha_3 \beta_4 (\theta_1 \theta_3),$
    \item \begin{equation*}
        \begin{split}
            X(\theta_1 \w m_2 (\theta_2) ) + m_3 X^t (\theta_1 \otimes \theta_2) &= \alpha_2 \beta_3 (\theta_1 \theta_2) - \theta_1 \alpha_1 \beta_2 (\theta_2) \\
            &\qquad - \beta_1 (\theta_1 ) \theta_2, \\
        \end{split}
    \end{equation*}
    \item \begin{equation*}
        \begin{split}
            X^t (\theta_2' \otimes m_2 (\theta_2)) + X^t (\theta_2 \otimes m_2( \theta_2') ) &= \alpha_3 \beta_4 (\theta_2 \theta_2') - \alpha_1 \beta_2 (\theta_2) \theta_2' \\
            &\qquad - \alpha_1 \beta_2 (\theta_2') \theta_2, \\
        \end{split}
    \end{equation*}
    \item $X^t ( \theta_1 \otimes X ( \theta_1'' \w \theta_1')) + X^t(\theta_1'' \otimes X( \theta_1 \w \theta_1')) = 0.$
\end{enumerate}
\end{prop}

\begin{proof}
$(1)$: 
\begingroup
\allowdisplaybreaks
\begin{align*}
        \beta_2 X(\theta_1 \w \theta_1' ) \cdot \phi_2 &= X (\theta_1 \w \theta_1' ) \cdot \alpha (\phi_2) \\
        &= (\beta_4^{-1} \circ h_4 ) (\theta_1 \ \theta_1' \otimes \alpha_2 (\phi_2) ) \\
        &= 0 \qquad \textrm{(by Corollary \ref{h1})} . \\
    \end{align*}
    \endgroup
$(2)$:
\begingroup
\allowdisplaybreaks
\begin{align*}
        \phi_1 \cdot \beta_3 X^t (\theta_1 \otimes \theta_2 ) &= - \alpha_1 (\phi_1) \cdot X^t (\theta_1 \otimes \theta_2) \\
        &= - (\beta_4^{-1} \circ h_4 ) (\alpha_1 (\phi_1) \w \theta_1 \otimes \theta_2 ) \\
        &= 0 \qquad \textrm{(by Corollary \ref{h1})}. \\
    \end{align*}
    \endgroup
$(3)$: For the first equality,
\begingroup
\allowdisplaybreaks
\begin{align*}
        \theta_1' \cdot X^t (\theta_1 \otimes \alpha_2 ( \phi_2 ) ) &= (\beta_4^{-1} \circ h_4)(\theta_1' \w \theta_1 \otimes \alpha_2 (\phi_2) ) \\
        &=0  \qquad \textrm{(by Corollary \ref{h1})}. 
    \end{align*}
    \endgroup
    For the second equality, multiply by an arbitrary $\theta \in M_1$:
    \begingroup\allowdisplaybreaks
    \begin{align*}
        \theta \alpha_1 ( \phi_1) \cdot X ( \theta_1 \w \theta_1') &= \beta_4^{-1} \circ h_3 (\theta_1 \w \theta_1' \otimes \theta \alpha_1 (\phi_1) ) \\
        &= 0 \qquad \textrm{(by Corollary \ref{h1})} .
    \end{align*}
    \endgroup
$(4)$:
\begingroup
\allowdisplaybreaks
\begin{align*}
        &\alpha_1 (\phi_1) \theta_1'' \cdot X (\theta_1 \w \theta_1') +\alpha_1 (\phi_1) \theta_1 \cdot X (\theta_1'' \w \theta_1') \\
        =& (\beta_4^{-1} \circ h_4) (\theta_1 \w \theta_1' \otimes \alpha_1 (\phi_1) \theta_1'' + \theta_1'' \w \theta_1' \otimes \alpha_1 (\phi_1) \theta_1 ) \\
        =& 0 \qquad \textrm{(by Corollary \ref{h1})}. 
   \end{align*}
    \endgroup
    To prove the additional claim, apply $m_4$ to the equality
    $$\alpha_1 (\phi_1) \theta_1 \cdot X (\theta_1 \w \theta_1') =0$$
    Recalling that $\alpha_0$ is the identity, one finds:
    \begingroup\allowdisplaybreaks
    \begin{align*}
        &k_1 ( \phi_1) \theta_1 X ( \theta_1 \w \theta_1') + \alpha_1 ( \phi_1) m_1(\theta_1) X( \theta_1 \w \theta_1') \\ 
        &+\alpha_1 ( \phi_1) \theta_1 m_2(X ( \theta_1 \w \theta_1')). 
    \end{align*}
    \endgroup
    Observe that $m_1 ( \theta_1) \alpha_1 ( \phi_1) X ( \theta_1 \w \theta_1') = 0$ by Property $(3)$. 
    
    We want to show that $\alpha_1 ( \phi_1) \theta_1 m_2 ( X ( \theta_1 \w \theta_1' ) ) = 0$. Multiplying by any $\alpha_1 ( \phi_1')$, one must obtain $0$ by property $(3)$. Applying $m_4$ and expanding using the Leibniz rule,
    \begingroup\allowdisplaybreaks
    \begin{align*}
        m_4 \Big( \alpha_1 ( \phi_1' ) \alpha_1 ( \phi_1) \theta_1 m_2 ( X (\theta_1 \w \theta_1') ) \Big) &= k_1( \phi_1 ' ) \alpha_1 ( \phi_1) \theta_1 m_2 ( X (\theta_1 \w \theta_1') ) \\
        &- k_1( \phi_1  ) \alpha_1 ( \phi_1') \theta_1 m_2 ( X (\theta_1 \w \theta_1') ) \\
        &+ m_1 ( \theta_1 ) \alpha_1 ( \phi_1') \alpha_1 ( \phi_1) m_2 ( X (\theta_1 \w \theta_1') ). 
    \end{align*}
    \endgroup
    The last term is $0$ by property $(3)$ combined with the Leibniz rule, whence the above shows that for all $\phi_1, \ \phi_1' \in K_1$,
    $$k_1( \phi_1 ' ) \alpha_1 ( \phi_1) \theta_1 m_2 ( X (\theta_1 \w \theta_1') ) = k_1( \phi_1  ) \alpha_1 ( \phi_1') \theta_1 m_2 ( X (\theta_1 \w \theta_1') )$$
    Since $\phi_1$ and $\phi_1'$ are totally arbitrary and $k_1 (K_1)$ has grade $\geq 2$, it follows that $\alpha_1 ( \phi_1 ) \theta_1 m_2 ( X ( \theta_1 \w \theta_1' ) ) = 0$. Combining this with the above, one finds $k_1 ( \phi_1) \theta_1 X ( \theta_1 \w \theta_1') = 0$. Since $k_1 ( \phi_1 )$ may be chosen to be regular, $\theta_1 X ( \theta_1 \w \theta_1') = 0$. Now let $\theta_1 \mapsto \theta_1 + \theta_1''$ to obtain the desired equality. \\

$(5)$: 
\begingroup
\allowdisplaybreaks
\begin{align*}
        m_2  X (\theta_1 \w \theta_1') \cdot \theta_3 &= -X(\theta_1 \w \theta_1' ) \cdot m_3 (\theta_3) \\
        &=-(\beta_4^{-1} \circ h_4) (\theta_1 \w \theta_1' \otimes m_3 (\theta_3) ) \\
        &= -(\beta_4^{-1} \circ h_4) (d (\theta_1 \w \theta_1' \otimes \theta_3) ) \\
        &= -\beta_4^{-1} \circ c_5 (\theta_1 \w \theta_1' \otimes \theta_3 ) \\
        &= \beta_4^{-1} \Big( \beta_1 (\theta_1' ) \beta_4 ( \theta_1 \theta_3 ) - \beta_1 (\theta_1 ) \beta_4 (\theta_1' \theta_3 ) \\
        &\qquad+ \beta_2 (\theta_1 \w \theta_1' ) \beta_3 (\theta_3) \Big) \\
        &=\Big(  \beta_1 (\theta_1' ) \theta_1 - \beta_1 (\theta_1) \theta_1' - \alpha_1 \beta_2 (\theta_1 \theta_1' ) \Big) \theta_3. 
    \end{align*}
    \endgroup
$(6)$: This follows from $(5)$ since:
\begingroup
\allowdisplaybreaks
\begin{align*}
        m_2  X (\theta_1 \w \theta_1') \cdot \theta_3 &= \theta_1' \cdot X^t (\theta_1 \otimes m_3 (\theta_3)) \\
        &=  \beta_1 (\theta_1' ) \theta_1 \theta_3 - \beta_1 (\theta_1) \theta_1' \theta_3 - \alpha_1 \beta_2 (\theta_1 \theta_1' ) \theta_3 \\
        &= - \theta_1' \alpha_3 \beta_4 (\theta_1 \theta_3 ) - \theta_1' \beta_1 (\theta_1 ) \theta_3 + \theta_1' \cdot \theta_1 \alpha_2 \beta_3 (\theta_3) \\
        &= \theta_1' \cdot \Big(  \theta_1 \alpha_2 \beta_3 (\theta_3) - \beta_1 (\theta_1) \theta_3 - \alpha_3 \beta_4 (\theta_1 \theta_3) \Big). 
   \end{align*}
    \endgroup
$(7)$: 
\begingroup
\allowdisplaybreaks
\begin{align*}
        &\Big( X(\theta_1 \w m_2 (\theta_2) ) + m_3 X^t (\theta_1 \otimes \theta_2)\Big) \cdot \theta_2' \\
        &= (\beta_4^{-1} \circ h_4) (\theta_1 \w m_2 (\theta_2 ) \otimes \theta_2' + \theta_1 \w m_2 (\theta_2') \otimes \theta_2 ) \\
        &= (\beta_4^{-1} \circ h_4 )(d (-\theta_1 \otimes \theta_2 \cdot \theta_2' )) \\
        &= \beta_4^{-1} \circ c_5 (-\theta_1 \otimes \theta_2 \cdot \theta_2' ) \\
        &= \beta_4^{-1} \Big(- \beta_1 (\theta_1) \beta_4 (\theta_2 \theta_2') + \beta_3 (\theta_1 \theta_2 ) \beta_2 (\theta_2') + \beta_3 (\theta_1 \theta_2') \beta_2 (\theta_2) \Big) \\ &= -\beta_1 (\theta_1) \theta_2 \theta_2' +  \alpha_2 \beta_3 (\theta_1 \theta_2 ) \theta_2' + \theta_1 \theta_2' \alpha_1 \beta_2 (\theta_2 ) \\
        &= \Big(- \beta_1 (\theta_1) \theta_2  +  \alpha_2 \beta_3 (\theta_1 \theta_2 )  - \theta_1'  \alpha_1 \beta_2 (\theta_2 ) \Big) \theta_2' . 
    \end{align*}
    \endgroup
$(8)$: This follows from $(7)$, since this is just the adjoint version.
\begingroup
\allowdisplaybreaks
\begin{align*}
      &\theta_1 \Big( X^t (\theta_2' \otimes m_2 (\theta_2)) + X^t (\theta_2 \otimes m_2( \theta_2') ) \Big) \\
      &=  -X(\theta_1 \w m_2(\theta_2) )\cdot \theta_2' - m_3 X^t (\theta_1 \w \theta_2 ) \cdot \theta_2' \\
      &= -\beta_1 (\theta_1) \theta_2 \theta_2' +  \alpha_2 \beta_3 (\theta_1 \theta_2 ) \theta_2' + \theta_1 \theta_2' \alpha_1 \beta_2 (\theta_2 ) \\
      &=+ \theta_1 \alpha_3 \beta_4 (\theta_2 \theta_2') - \theta_1 \theta_2 \alpha_1 \beta_2 (\theta_2') - \theta_1 \theta_2' \alpha_1 \beta_2 (\theta_2) \\
      &= \theta_1 \cdot \Big(  \alpha_3 \beta_4 (\theta_2 \theta_2') -  \theta_2 \alpha_1 \beta_2 (\theta_2') -  \theta_2' \alpha_1 \beta_2 (\theta_2) \Big). 
   \end{align*}
    \endgroup
    $(9)$: Observe that it suffices to show $X^t( \theta_1 \otimes X(\theta_1 \w \theta_1')) = 0$, since one may then substitute $\theta_1 \mapsto \theta_1 + \theta_1''$ to obtain the general case. Multiplying by an arbitrary $\theta \in M_1$, one obtains
    $$X( \theta_1 \w \theta ) X( \theta_1 \w \theta_1'),$$
    so it suffices to show this product is $0$. Since this is an element of $M_4$ and $m_4$ is injective, it suffices to show that $m_4$ applied to the above is $0$. One computes:
    \begingroup\allowdisplaybreaks
    \begin{align*}
        &m_4(X( \theta_1 \w \theta ) X( \theta_1 \w \theta_1')) \\
        =& m_2(X( \theta_1 \w \theta )) X( \theta_1 \w \theta_1') + X( \theta_1 \w \theta ) m_2(X( \theta_1 \w \theta_1')) \\
        =&\big( \beta_1 (\theta ) \theta_1 - \beta_1 (\theta_1) \theta - \alpha_1 \beta_2 (\theta_1 \theta ) \big) X( \theta_1 \w \theta_1') \\
        &+ \big( \beta_1 (\theta_1' ) \theta_1 - \beta_1 (\theta_1) \theta_1' - \alpha_1 \beta_2 (\theta_1 \theta_1' ) \big) X( \theta_1 \w \theta) \qquad \textrm{(by property} \ (5))\\
        =& \big( \beta_1 (\theta ) \theta_1 - \beta_1 (\theta_1) \theta \big) X( \theta_1 \w \theta_1') \\
        &+ (\beta_1 (\theta_1' ) \theta_1 - \beta_1 (\theta_1) \theta_1' ) X( \theta_1 \w \theta)\qquad \textrm{(by property} \ (2))\\
        =& - \beta_1 ( \theta_1) \big( \theta X( \theta_1 \w \theta_1') + \theta_1' X(\theta_1 \w \theta) \big)  \qquad \textrm{(by property} \ (4)) \\
        =& 0 \qquad \textrm{(again, by property} \ (4)).
    \end{align*}
    \endgroup
\end{proof}

\section{The Length 4 Big From Small Construction is a DG $R$-Algebra}\label{itsdg}

In this section, we prove the main result of the paper. Theorem \ref{thm:BFSisDG} states that the length $4$ big from small construction as in Definition \ref{bfs} admits the structure of a commutative, associative DG$\Gamma$ $R$-algebra exhibiting Poincar\'e duality. The method of proof is to write down an explicit product and directly verify that the properties of Definition \ref{dg} are satisfied. The product given in Theorem \ref{thm:BFSisDG} is inspired by a similar product used in \cite{minres}.

It is also worth noting that Theorem \ref{thm:BFSisDG} holds even if one only assumes in Setup \ref{setup2} that the complex $M$ is a length $4$ resolution of a cyclic $R$-module with $\rank (M_4) = 1$ and $M \cong \hom_R (M , M_4)$, since Theorem \ref{l4resdg} may be employed.

\begin{theorem}\label{thm:BFSisDG}
Adopt notation and hypotheses as in Setup \ref{setup2}. Then the complex $F(\alpha,r)$ of Definition \ref{bfs} admits the structure of a commutative associative DG $R$-algebra exhibiting Poincar\'e duality via the following multiplication:
\begingroup
\allowdisplaybreaks
\begin{align*}
    F_1 \otimes F_1 &\to F_2 \\
    \begin{pmatrix} \phi_1 \\
    \theta_1 \\
    \end{pmatrix} \begin{pmatrix} \phi_1' \\
    \theta_1' \\
    \end{pmatrix} &= \begin{pmatrix} \phi_1 \phi_1' \\
    - \alpha_1 (\phi_1) \theta_1' - \theta_1 \alpha_1 ( \phi_1') - r \theta_1 \theta_1' + X( \theta_1 \w \theta_1')  \\
    \alpha_1 ( \theta_1 ) \phi_1' - \alpha_1 (\phi_1') \theta_1 + \beta_2 (\theta_1 \theta_1 ' ) \\
    \end{pmatrix} \\
    F_1 \otimes F_2 &\to F_3 \\
    \begin{pmatrix} \phi_1 \\
    \theta_1 \\
    \end{pmatrix} \begin{pmatrix} \phi_2 \\
    \theta_2 \\
    \phi_1' \\
    \end{pmatrix} &= \begin{pmatrix} \theta_1\alpha_2 (\phi_2) - [\phi_1 \phi_2]_K \alpha_4 (h) - \alpha_1 (\phi_1) \theta_2 - r \theta_1 \theta_2 + X^t (\theta_2 \otimes \theta_1) \\
    \phi_1 \phi_1' - m_1 (\theta_1) \phi_2 - \beta_3 (\theta_1 \theta_2) \\
    \end{pmatrix} \\
    F_1 \otimes F_3 &\to F_4 \\
    \begin{pmatrix} \phi_1 \\
    \theta_1 \\
    \end{pmatrix} \begin{pmatrix} \theta_3 \\
    \phi_2 \\
    \end{pmatrix} &= [\phi_1 \phi_2 ]_K h - \theta_1 \theta_3 \\
    F_2 \otimes F_2 &\to F_4 \\
    \begin{pmatrix} \phi_2 \\
    \theta_2 \\
    \phi_1 \\
    \end{pmatrix} \begin{pmatrix} \phi_2' \\
    \theta_2' \\
    \phi_1' \\
    \end{pmatrix} &= [\phi_2 \phi_1' ]_K h + [\phi_1 \phi_2' ]_K h - \theta_2 \theta_2' ,
\end{align*}
\endgroup
where $h \in M_4$ is such that $[h]_M = 1$ and the maps $X$ and $X^t$ are defined in Definition \ref{Xnew}.
\end{theorem}

\begin{proof}
We first show that associatvity holds for $3$ elements of degree $1$. Consider the following associativity term:
\begingroup\allowdisplaybreaks
\begin{align*}
    \begin{pmatrix} \phi_1 \\
    \theta_1 \\
    \end{pmatrix} \Bigg( \begin{pmatrix} \phi_1' \\
    \theta_1' \\
    \end{pmatrix} \begin{pmatrix} \phi_1'' \\
    \theta_1'' \\
    \end{pmatrix} \Bigg) - \begin{pmatrix} \phi_1'' \\
    \theta_1'' \\
    \end{pmatrix} \Bigg( \begin{pmatrix} \phi_1 \\
    \theta_1 \\
    \end{pmatrix} \begin{pmatrix} \phi_1' \\
    \theta_1' \\
    \end{pmatrix} \Bigg).
\end{align*}
\endgroup
Let us first compute the top entry of the above expression:
\begingroup\allowdisplaybreaks
\begin{align*}
    &\theta_{{1}}\alpha \left( \phi_1'\phi_1'' \right) - \left( 
\theta_1''\alpha \left( \phi_1' \right) -\theta_1'\alpha
 \left( \phi_1'' \right) -r\theta_1'\theta_1''+X \left( 
\theta_1'\w \theta_1'' \right)  \right) \alpha \left( \phi_{{1}}
 \right) \\
 &-r\theta_{{1}} \left( \theta_1''\alpha \left( \phi_1'
 \right) -\theta_1'\alpha \left( \phi_1'' \right) -r\theta_1'\theta_1''+X \left( \theta_1' \w \theta_1'' \right) 
 \right) \\
 &+{\it X^t} \left( \theta_{{1}} \otimes \left( \theta_1''\alpha
 \left( \phi_1' \right) -\theta_1'\alpha \left( \phi_1''
 \right) -r\theta_1'\theta_1''+X \left( \theta_1' \w\theta_1'' \right)  \right)  \right) \\
&-\theta_1''\alpha \left( \phi_{{1}
}\phi_1' \right) + \left( \theta_1'\alpha \left( \phi_{{1}}
 \right) -\theta_{{1}}\alpha \left( \phi_1' \right) -r\theta_{{1}}
\theta_1'+X \left( \theta_{{1}} \w \theta_1' \right)  \right) 
\alpha \left( \phi_1'' \right) \\
&+r\theta_1'' \left( \theta_1'\alpha \left( \phi_{{1}} \right) -\theta_{{1}}\alpha \left( \phi_1' \right) -r\theta_{{1}}\theta_1'+X \left( \theta_{{1}} \w \theta_1' \right)  \right)\\
& -{\it X^t} \left( \theta_1'' \otimes \left( \theta_1'\alpha \left( \phi_{{1}} \right) -\theta_{{1}}\alpha \left( \phi_1' \right) -r\theta_{{1}}\theta_1'+X \left( \theta_{{1}}\w \theta_1' \right)  \right)  \right). 
\end{align*}
\endgroup
After cancelling off the easy terms, one is left with:
\begingroup\allowdisplaybreaks
\begin{align*}
    &-X \left( 
\theta_1'\w \theta_1'' \right)  \alpha \left( \phi_{{1}}
 \right) \\
 &-r\theta_{{1}} X \left( \theta_1' \w \theta_1'' \right)  \\
 &+{\it X^t} \left( \theta_{{1}} \otimes \left( \theta_1''\alpha
 \left( \phi_1' \right) -\theta_1'\alpha \left( \phi_1''
 \right) -r\theta_1'\theta_1''+X \left( \theta_1' \w\theta_1'' \right)  \right)  \right) \\
&+X \left( \theta_{{1}} \w \theta_1' \right) 
\alpha \left( \phi_1'' \right) \\
&+r\theta_1'' X \left( \theta_{{1}} \w \theta_1' \right) \\
& -{\it X^t} \left( \theta_1'' \otimes \left( \theta_1'\alpha \left( \phi_{{1}} \right) -\theta_{{1}}\alpha \left( \phi_1' \right) -r\theta_{{1}}\theta_1'+X \left( \theta_{{1}}\w \theta_1' \right)  \right)  \right). 
\end{align*}
\endgroup
Notice that
$$X^t ( \theta_1 \otimes \theta_1'' \alpha ( \phi_1') + \theta_1'' \otimes \theta_1 \alpha ( \phi_1')) = 0,$$
since after taking the product with any other $\theta \in M_1$, one finds
$$-\beta_4^{-1} \circ h_4 ( \theta_1 \w \theta \otimes \theta_1'' \alpha ( \phi_1') + \theta_1'' \w \theta \otimes \theta_1 \alpha ( \phi_1')),$$
and this is $0$ by Corollary \ref{h1}. Two other terms like this cancel in a similar fashion. One then has:
\begingroup\allowdisplaybreaks
\begin{align*}
    &-X \left( 
\theta_1'\w \theta_1'' \right)  \alpha \left( \phi_{{1}}
 \right) + X \left( \theta_{{1}} \w \theta_1' \right) 
\alpha \left( \phi_1'' \right) \\
 &-r\theta_{{1}} X \left( \theta_1' \w \theta_1'' \right) +r\theta_1'' X \left( \theta_{{1}} \w \theta_1' \right) \\
 &+{\it X^t} \left( \theta_1 \otimes X \left( \theta_1' \w\theta_1'' \right)  \right) -{\it X^t} \left( \theta_1'' \otimes X \left( \theta_{{1}}\w \theta_1' \right)  \right). 
\end{align*}
\endgroup
After multiplying by an arbitrary $\theta \in M_1$, the first term of the top line is:
$$-\beta_4^{-1} \circ h_4 ( \theta_1' \w \theta_1'' \otimes \theta \alpha( \phi_1) ).$$
This is $0$ by property $(3)$ of Proposition \ref{properties}. Similarly for the second term. For the middle expression, this vanishes by Property $(4)$ of Proposition \ref{properties}, and the final $2$ terms vanish by Property $(9)$ of \ref{properties}.

The expression in the bottom entry is then computed as:
\begingroup\allowdisplaybreaks
\begin{align*}
    &\phi_{{1}} \left( m \left( \theta_1' \right) \phi_1''-m
 \left( \theta_1'' \right) \phi_1'+\beta \left( \theta_1'
\theta_1'' \right)  \right) \\
&-m \left( \theta_{{1}} \right) \phi_1'\phi_1''-\beta \left( \theta_{{1}} \left( \theta_1''\alpha
 \left( \phi_1' \right) -\theta_1'\alpha \left( \phi_1''
 \right) -r\theta_1'\theta_1''+X \left( \theta_1' \w \theta_1'' \right)  \right)  \right) \\
 &-\phi_1'' \left( m \left( \theta_{{1}} \right) \phi_1'-m \left( \theta_1' \right) \phi_{{1}}+
\beta \left( \theta_{{1}}\theta_1' \right)  \right)\\
&+m \left( \theta_1'' \right) \phi_{{1}}\phi_1'+\beta \left( \theta_1'' \left( \theta_1'\alpha \left( \phi_{{1}} \right) -\theta_{{1}}
\alpha \left( \phi_1' \right) -r\theta_{{1}}\theta_1'+X
 \left( \theta_{{1}} \w \theta_1' \right)  \right)  \right). 
\end{align*}
\endgroup
Recall that $\beta \circ X = 0$, so all $X$ terms vanish. Using that $\beta ( \alpha ( \phi_i ) \theta_j ) = \phi_i \beta ( \theta_j)$, more terms cancel in this way. The rest of the terms cancel without any special properties.

For the next associativity term, take the product of $2$ elements in degree $1$ and $1$ element in degree $2$ to obtain:
\begingroup\allowdisplaybreaks
\begin{align*}
&[ \phi_{{1}} \left( \phi_1'\phi_1''-m \left( \theta_1' \right) \phi_{{2}}-\beta \left( \theta_1'\theta_{{2}}
 \right)  \right) ] h \\
 &-\theta_{{1}} \left( \theta_1'\alpha
 \left( \phi_{{2}} \right) -[\phi_1'\phi_{{2}} ] m
 \left( h \right) -\theta_{{2}}\alpha \left( \phi_1' \right) -r
\theta_1'\theta_{{2}}+{\it X^t} \left( \theta_1'\theta_{{2}}
 \right)  \right) \\
 &-[ \phi_{{2}} \left( m \left( \theta_{{1}}
 \right) \phi_1'-m \left( \theta_1' \right) \phi_{{1}}+\beta
 \left( \theta_{{1}}\theta_1' \right)  \right)  ] h-[ \phi_1''\phi_{{1}}\phi_1' ] h\\
 &+\theta_{{2}}
 \left( \theta_1'\alpha \left( \phi_{{1}} \right) -\theta_{{1}}
\alpha \left( \phi_1' \right) -r\theta_{{1}}\theta_1'+X
 \left( \theta_{{1}}\theta_1' \right)  \right). 
 \end{align*}
 \endgroup
 For the nontrivial cancellations, observe first that $[ \phi_1 \beta ( \theta_1' \theta_2) ] h = \theta_2 \theta_1' \alpha( \phi_1)$; this cancels with the first term on the bottom line. A second term cancels similarly. There are $4$ terms leftover after all trivial cancellations. Firstly, there is the expression
 $$- \theta_1 X^t ( \theta_1' \otimes \theta_2) + \theta_2 X ( \theta_1 \w \theta_1'),$$
 but this is zero by definition. Lastly, one has
 $$[\phi_1' \phi_2] \theta_1 m( h) - [ \phi_2 m ( \theta_1) \phi_1' ] h.$$
 Since $M_5 = 0$, $\theta_1 m(h) = m( \theta_1) h$, so the above is $0$. This proves associativity.
 
 For the Leibniz rule on two elements of degree $1$, the top entry is
 \begingroup\allowdisplaybreaks
 \begin{align*}
     &k(\phi_1 \phi_1') + \beta (\theta_1' \alpha ( \phi_1)) - \beta ( \theta_1 \alpha ( \phi_1')) - r \beta ( \theta_1 \theta_1') \\
     &+ \beta (X ( \theta_1 \theta_1')) + r m( \theta_1) \phi_1' - r m( \theta_1' ) \phi_1 + r \beta ( \theta_1 \theta_1' ) \\
     &-k ( \phi_1) \phi_1' - \beta ( \theta_1) \phi_1' -r m( \theta_1) \phi_1'\\ 
     &+ k ( \phi_1') \phi_1 + \beta ( \theta_1') \phi_1 +r m( \theta_1') \phi_1.
 \end{align*}
 \endgroup
 Recalling that $\beta \circ X = 0$, all other terms cancel trivially. For the bottom entry, this is computed as:
 \begingroup\allowdisplaybreaks
 \begin{align*}
     &-m (\theta_1' \alpha ( \phi_1) ) + m ( \theta_1 \alpha ( \phi_1')) + r m( \theta_1 \theta_1') \\
     &-m (X( \theta_1 \theta_1')) - \alpha ( m( \theta_1) \phi_1') + \alpha ( m( \theta_1' ) \phi_1) + \alpha ( \beta ( \theta_1 \theta_1') ) \\
     &-k ( \phi_1) \theta_1' - \beta ( \theta_1) \theta_1' -r m( \theta_1) \theta_1'\\ 
     &+ k ( \phi_1') \theta_1 + \beta ( \theta_1')\theta_1 +r m( \theta_1') \theta_1\\
     =&-m ( X ( \theta_1 \w \theta_1' ) + \alpha( \beta ( \theta_1 \theta_1')) - \beta ( \theta_1) \theta_1' + \beta( \theta_1') \theta_1 . 
 \end{align*}
 \endgroup
 By property $(5)$ of Proposition \ref{properties}, this is $0$. For the Leibniz rule on elements of degree $1$ and $2$, the top entry is computed as:
 \begingroup\allowdisplaybreaks
 \begin{align*}
     &\beta ( \theta_1 \alpha ( \phi_2) ) - [\phi_1 \phi_2] \beta ( m (h)) - \beta ( \theta_2 \alpha ( \phi_1)) - r \beta ( \theta_1 \theta_2) \\
     &+ \beta (X^t ( \theta_1 \theta_2) - r \phi_1 \phi_1' + r m( \theta_1) \phi_2 + r \beta ( \theta_1 \theta_2) \\
     &-k ( \phi_1) \phi_2 - \beta ( \theta_1) \phi_2 -r m( \theta_1) \phi_2\\
     &+ \phi_1 k ( \phi_2) + \phi_1 \beta ( \theta_2) + r \phi_1 \phi_1'. 
 \end{align*}
 \endgroup
 Almost all terms cancel trivially. Recall that $[\phi_1 \phi_2] \beta ( m(h)) = k(\phi_1 \phi_2)$ to see that these terms cancel. The middle entry is
 \begingroup\allowdisplaybreaks
 \begin{align*}
     &-m ( \theta_1 \alpha ( \phi_2) ) + [\phi_1 \phi_2] m ( m (h)) + m ( \theta_2 \alpha ( \phi_1)) + r m ( \theta_1 \theta_2) \\
     &- m (X^t ( \theta_1 \otimes \theta_2) )- \alpha( \phi_1 \phi_1') + \alpha( m( \theta_1) \phi_2) + \alpha( \beta ( \theta_1 \theta_2)) \\
     &-k ( \phi_1) \theta_2 - \beta ( \theta_1) \theta_2 -r m( \theta_1) \theta_2\\
     &- m( \theta_2) \alpha ( \phi_1) - \alpha ( \phi_1') \alpha ( \phi_1) - \theta_1 \alpha ( k ( \phi_2)) - \theta_1 \alpha ( \beta ( \theta_2)) - r \theta_1 \alpha ( \phi_1') \\
     &+r \theta_1 m( \theta_2) + r \theta_1 \alpha ( \phi_1') - X ( \theta_1 \w m(\theta_2) ) - X( \theta_1 \w \alpha( \phi_1') ) \\
     =& - m ( X^t ( \theta_1 \otimes  \theta_2)) + \alpha ( \beta ( \theta_1 \theta_2)) \\
     &-\beta ( \theta_1 ) \theta_2 - \theta_1 \alpha ( \beta ( \theta_2)) - X( \theta_1 \w \m( \theta_2)). 
 \end{align*}
 \endgroup
 By Property $(7)$ of Proposition \ref{properties}, this is $0$. The bottom entry is:
 \begingroup\allowdisplaybreaks
 \begin{align*}
     &k(\phi_1 \phi_1') - m( \theta_1) k ( \phi_2) - k ( \beta ( \theta_1 \theta_2) \\
     &-k ( \phi_1) \phi_1' - \beta ( \theta_1) \phi_1' -r m( \theta_1) \phi_1'\\
     &m( \theta_1) k( \phi_2) + m ( \theta_1) \beta ( \theta_2) + r m( \theta_1) \phi_1' + m( \alpha (\phi_1') ) \phi_1 \\
     &- \beta ( \theta_1 m( \theta_2) ) - \beta ( \theta_1 \alpha ( \phi_1')), 
 \end{align*}
 \endgroup
 and these terms cancel without any additional properties. For the Leibniz rule on elements of degree $1$ and $3$, the top entry is:
 \begingroup\allowdisplaybreaks
 \begin{align*}
     &[\phi_1 \phi_2] \alpha ( \beta (h)) - \alpha ( \beta ( \theta_1 \theta_3)) - r[\phi_1 \phi_2] m(h) + r m( \theta_1 \theta_2) \\
     &-k ( \phi_1) \theta_3 - \beta ( \theta_1) \theta_3 -r m( \theta_1) \theta_3\\
     &+ \theta_1 \alpha ( \beta ( \theta_3)) - r \theta_1 \alpha ( \phi_2) - [\phi_1 \beta (\theta_3) - r \phi_1 \phi_2 ] m ( h) \\
     &+m( \theta_3) \alpha ( \phi_1) + \alpha ( \phi_2) \alpha ( \phi_1) + r \theta_1 m( \theta_3) + r \theta_1 \alpha ( \phi_2) \\
     &- X^t ( \theta_1 \otimes m ( \theta_3)) - X^t ( \theta_1 \otimes \alpha (\phi_2)) \\
     =& -\alpha ( \beta ( \theta_1 \theta_3)) - \beta ( \theta_1 ) \theta_3 + \theta_1 \alpha ( \beta ( \theta_3)) \\
     &-X^t ( \theta_1 \otimes m( \theta_3) ) .
 \end{align*}
 \endgroup
 By property $(6)$ of Proposition \ref{properties}, this is $0$. The bottom entry is:
 \begingroup\allowdisplaybreaks
 \begin{align*}
     &-[\phi_1 \phi_2] k(\beta(h)) + k ( \beta(\theta_1 \theta_3)) \\
     &-k ( \phi_1) \phi_2 - \beta ( \theta_1) \phi_2 -r m( \theta_1) \phi_2\\
     &+ \phi_1 k ( \phi_2) - m(\theta_1) \beta ( \theta_3) + r m( \theta_1) \phi_2 \\
     &+ \beta ( \theta_1 m( \theta_3) ) + \beta (\theta_1  \alpha ( \phi_2)). 
 \end{align*}
 \endgroup
 Again, recalling that $\beta ( \alpha ( \phi_1) \theta_1) = \phi_1 \beta ( \theta_1)$, the above terms cancel easily. 
 
 For the Leibniz rule on elements of degree $1$ and $4$:
 \begingroup\allowdisplaybreaks
 \begin{align*}
     &-k ( \phi_1) \theta_4 - \beta ( \theta_1) \theta_4 -r m( \theta_1) \theta_4\\
     &-[\phi_1 k(\beta(\theta_4))] h - \theta_1 \alpha(\beta(\theta_4)) + r \theta_1 m( \theta_4). 
 \end{align*}
 \endgroup
 Notice $[\phi_1 k \beta ( \theta_4) ]h =- k(\phi_1) [ \theta_4] h = -k( \phi_1) \theta_4$ and $\theta_1 \alpha ( \beta ( \theta_4) ) = \beta ( \theta_1 ) \theta_4$, so these terms cancel. For the Leibniz rule on elements both of degree $2$, the top entry is:
 \begingroup\allowdisplaybreaks
 \begin{align*}
     &[\phi_1' \phi_2]\alpha(\beta( h)) + [\phi_1 \phi_2' ] \alpha(\beta(h)) - \alpha(\beta(\theta_2 \theta_2')) \\
     &-r[\phi_1' \phi_2]m(h) - [\phi_1 \phi_2' ] rm(h) +rm(\theta_2 \theta_2') \\
     &+m(\theta_2) \alpha ( \phi_2') + \alpha ( \phi_1) \alpha ( \phi_2') \\
     &+ [r k (\phi_2) \phi_2'+ \beta ( \theta_2) \phi_2' + r \phi_1 \phi_2' ] m(h) + \theta_2' \alpha ( k(\phi_2)) \\
     &+ \theta_2' \alpha ( \beta ( \theta_2)) + r \theta_2' \alpha ( \phi_1) - r m( \theta_2) \theta_2' - r\alpha ( \phi_1) \theta_2' \\
     &+X^t ( m ( \theta_2) \otimes \theta_2') + X^t ( \alpha(\phi_1) \otimes \theta_2') \\
     &+m(\theta_2') \alpha ( \phi_2) + \alpha ( \phi_1') \alpha ( \phi_2) \\
     &+ [r k (\phi_2') \phi_2+ \beta ( \theta_2') \phi_2 + r \phi_1' \phi_2 ] m(h) + \theta_2 \alpha ( k(\phi_2')) \\
     &+ \theta_2 \alpha ( \beta ( \theta_2')) + r \theta_2 \alpha ( \phi_1') - r m( \theta_2') \theta_2 - r\alpha ( \phi_1') \theta_2 \\
     &+X^t ( m ( \theta_2') \otimes \theta_2) + X^t ( \alpha(\phi_1') \otimes \theta_2) \\
     =&- \alpha ( \beta ( \theta_2 \theta_2' )) + \theta_2' \alpha ( \beta ( \theta_2) ) + X^t ( m( \theta_2 ) \otimes \theta_2') \\
     &+ \theta_2 \alpha ( \beta ( \theta_2')) + X^t ( m ( \theta_2') \otimes \theta_2). 
 \end{align*}
 \endgroup
 By property $(8)$ of Proposition \ref{properties}, this final term is $0$. The bottom entry is:
 \begingroup\allowdisplaybreaks
 \begin{align*}
     &-[\phi_1' \phi_2] k ( \beta(h)) -[\phi_1 \phi_2'] k ( \beta(h)) + k ( \beta ( \theta_2 \theta_2')) \\
     &-k(\phi_2) \phi_1' - \beta( \theta_2) \phi_1' - r \phi_1 \phi_1' \\
     &- m(\alpha(\phi_1))\phi_2' -\beta(m(\theta_2) \theta_2') - \beta(\alpha(\phi_1)\theta_2') \\
     &-k(\phi_2') \phi_1 - \beta( \theta_2') \phi_1 - r \phi_1' \phi_1 \\
     &- m(\alpha(\phi_1'))\phi_2 -\beta(m(\theta_2') \theta_2) - \beta(\alpha(\phi_1')\theta_2), 
 \end{align*}
 \endgroup
and everything cancels trivially, keeping in mind that $[\phi_1 \phi_2'] k ( \beta ( h)) = k( \phi_1 \phi_2')$. For the Leibniz rule on elements of degree $2$ and $3$:
\begingroup\allowdisplaybreaks
\begin{align*}
    &[k(\phi_2) \phi_2' + \beta( \theta_2) \phi_2' + r \phi_1 \phi_2' ] h\\
    &+ m( \theta_2) \theta_3 + \alpha( \phi_1) \theta_3 \\
    &+[\phi_2 k( \phi_2') ] h + [\phi_1 \beta( \theta_3) - r \phi_1 \phi_2'] h \\
    &+ \theta_2 m( \theta_3) + \theta_2 \alpha ( \phi_2'), 
\end{align*}
\endgroup
and again, all of these terms cancel trivially. This concludes the proof.
\end{proof}

\end{document}